\documentclass{amsart}
\usepackage{enumerate}

\keywords{Liv\v{s}ic Theorem}
\author{Genady Ya. Grabarnik}
\address{Dept. of Math \& Computer Science, St. John's University, Queens, NY, USA}
\email{grabarng@stjohns.edu}
\author{Misha Guysinsky} 
\address{Deptartment of Mathematics, The Pennsilvania State University, University Park,  PA, USA}
\email{guysin\_m@math.psu.edu}

\date{}

\title{ Liv\v{s}ic Theorem for Banach Rings }

\newtheorem{theorem}{Theorem}
\newtheorem{corollary}[theorem]{Corollary}
\newtheorem*{main}{Main Theorem}
\newtheorem*{lemman}{Lemma}
\newtheorem{lemma}[theorem]{Lemma}
\newtheorem{prop}[theorem]{Proposition}

\theoremstyle{definition}
\newtheorem{definition}[theorem]{Definition}

\newcommand{\la}{\ensuremath{\lambda}}
\newcommand{\ee}{\ensuremath{\varepsilon}}

\newcommand{\ep}{\ensuremath{\epsilon}}

\newcommand{\si}{\ensuremath{\mathbf{\sigma}}}
\newcommand{\mr}{\ensuremath{\mathbb{R}}}
\newcommand{\mz}{\ensuremath{\mathbb{Z}}}
\newcommand{\mf}{\ensuremath{\mathbb{F}}}

\newcommand{\dist}{\text{dist}}

\begin{document}

\begin{abstract} We prove the Liv\v{s}ic Theorem for  H\"{o}lder continuous cocycles with values in Banach rings. We consider   a transitive homeomorphism ${\si:X\to X}$  that satisfies the Anosov Closing Lemma, and  a H\"{o}lder continuous map ${a:X\to B^\times}$ from a compact metric space $X$  to the set of invertible elements of some Banach ring $B$. We show that it is a coboundary with  a H\"{o}lder continuous transition function if and only if ${a(\si^{n-1}p)\ldots a(\si p)a(p)=e}$ for each periodic point $p=\si^n p$.
\end{abstract}

\maketitle

\section{Introduction}
 We assume that $X$ is a compact metric space, $G$  a complete metric group,  and $\si:X\to X$  a homeomorphism. 

We say that  a map $a:\mathbb{Z}\times X\to G$ is {\it a cocycle} over  \si\  if 
$$a(n,x)=a(n-k,\si^kx)a(k,x)\quad\text{for any }n,k\in\mz$$

 Every map $a:X\to G$ generates a cocycle $a(n,x)$ defined as $$a(n,x)=a(\si^{n-1}x)a(\si^{n-2}x)\ldots a(x) \quad n>0$$
$$a(0,x)=Id$$
$$a(n,x)= a^{-1}(\si^{n}x)\ldots a^{-1}(\si^{-2}x)a^{-1}(\si^{-1}x)\quad n<0$$
We see that $a(1,x)=a(x)$. In this paper we consider only cocycles generated by H\"older continuous maps $a:X\to G$.

We say that a H\"older continuous map $a:X\to G$   is a {\it coboundary } (or more precisely generates a cocycle which is a coboundary) if there is a H\"older continuous function $t:X\to G$ such that
$$ a(x)=t(\si x)t^{-1}(x)$$
The function $t(x)$ is a called a {transition map}.
If $a(x)$ is a coboundary then it is clear that
$$a(n,x)=t(\si^n x)t^{-1}(x)$$
A question whether some cocycle is a coboundary or not appears naturally in many important problems in dynamical systems. 
There is  a simple necessary condition for a cocycle to be a coboundary. If $a(x)$ is a coboundary and $p\in X$ is  a periodic point  $\si^n p=p$  then
$$a(\si^{n-1}p)\ldots a(\si p)a(p)=a(n,p)=t(\si^n p)t^{-1}(p)=e$$
where $e$ is the identity element in the group $G$.

We say that for a cocycle $a(n,x)$ {\it periodic obstruction vanish} if
 \begin{equation}\label{e0} a(\si^{n-1}p)\ldots a(\si p)a(p)=e\quad\forall p\in X\text{ with } \si^np=p, n\in\mathbb{N} \end{equation}

 A.Liv\v{s}ic (see ~\cite{L1,L2})  proved that  when \si\ is a transitive Anosov map and the group $G$ is $\mathbb{R}$ or $\mathbb{R}^n$  then a cocycle  $a(x)$ is a coboundary if and only if the periodic obstruction vanish. This result is called Liv\v{s}ic theorem. The proof of  the Liv\v{s}ic theorem for other groups turned out to be harder.  Nevertheless, in the last twenty years in the series of papers (see \cite{BN},\cite{PW},\cite{P},\cite{KS},\cite{NT},\cite{LW}) it was shown that for some  groups under an additional assumption on the growth rates of the cocycle $a(n,x)$  the condition    (\ref{e0}) is also sufficient.    For example, in \cite{BN} it was shown that if $G=B^\times$ the set of invertible elements of some Banach algebra then if periodic obstruction vanish and $a(x)$ is close to the identity element $e$ then it is a coboundary. 
 The question remained if this additional  assumption will follow from the fact that the products along periodic points are equal to $e$. In 2011 B.Kalinin in \cite{Ka} made a breakthrough by proving the  Liv\v{s}ic theorem for functions with values in $GL(n,\mathbb{R})$ and  more generally in a connected Lie group assuming only that condition (\ref{e0}) is satisfied.
 
  He used  Lyapunov exponents for different invariant measures  to estimate the rate of the cocycle growth   and then  approximated Lyapunov exponents for all invariant measures by Lyapunov exponents only at periodic points.  To do the latter the Oseledets Theorem was used.  In this paper, we are  proving that  a cocycle with values in invertible elements of Banach ring is a coboundary if and only if periodic obstructions vanish. There is no analogs of the Oseledets Theorem for Banach rings ( or even Banach algebras). Still we can define analogs of the highest and lowest Lyapunov exponents and using a different  argument  show that they could be approximated by the values of the cocycle at periodic points. Examples of Banach rings include, Banach algebras, and Banach algebras with $\mf$ as a field of scalars, where $\mf$ is a local field. For them it is a new result. Also several  already known results follow: Liv\v{s}ic Theorem for cocycles with values  in $GL(n,\mr)$ (see \cite{Ka}) and $GL(n,\mf)$ (see \cite{LZ}). 
  
  As in \cite{Ka} we require that the map \si\ was transitive and had the following property  
 \begin{definition} We say that a homeomorphism $\si:X\to X$ has a {\it closing property} if there exist positive numbers $\delta_0, \la,C$ such that for any $x\in X$ and $n>0$ with $\dist(x,\si^n x)\le \delta_0$ we can find points $p,z\in X$ where
 $$\si^n p=p$$
 and  for every $i=0,1,\ldots, n$
 $$\dist(\si^i p,\si^i z)\le e^{-i\la}C\dist(x,\si^n x)\quad \dist(\si^i x,\si^i z)\le e^{-(n-i)\la}C\dist(x,\si^n x) $$
We will call $\lambda$ the expansion constant for the map \si.
\end{definition}
 Anosov maps and shifts of finite types are main examples of maps with closing property.
 \begin{definition} An associative (non--commutative) ring $B$  with the unity element $e$ is called  {\it Banach ring} if there is a function $\|\cdot\|:B\to\mr$ such that
 \begin{enumerate}
\item $\|a\|\ge 0$ and $\|a\|=0$ if and only if $a=0$.
\item $\|a+b\|\le \|a\|+\|b\|$.
\item $\|a\cdot b\|\le \|a\|\cdot \|b\|$.
\item The ring $B$ is a complete metric space with respect to the distance defined as $dist(a,b)=\|a-b\|$.
\end{enumerate}
 \end{definition}
 We denote as $B^\times$ the set of invertible elements of a Banach ring $B$. The main result of this paper is:
 \begin{main}\label{t2} Let $X$ be a compact metric space, $\si:X\to X$  a transitive homeomorphism with closing property. If $a:X\to B^\times $ is an  $\alpha$-H\"older continuous function such that 
 $$    a(\si^{n-1}p)\ldots a(\si p)a(p)=e\quad  \forall p\in X, n\in \mathbb{N} \text{ with } \si^np=p$$
  then there exists an $\alpha$-H\"older function $t:X\to B^\times$ such that
  $$ a(x)=t(\si x)t^{-1}(x)$$
 \end{main}

\section{Subadditive Cocycles}

Let $\si:X\to X$ be a continuous function. We will call a continuous function $s(n,x):\mathbb{Z}\times X\to \mathbb{R}$ \textit{ a subadditive cocycle} over the function $\si$  if 
$$s(n+m,x)\le s(n,\si^m x)+s(m,\si x)$$

From the Kingman's Theorem about subadditive cocycles  \cite{Ki, Furstenberg} follows that for every $\si$-invariant   measure $\mu$ and for almost all $x$ there exists a number 
\begin{equation}\label{e1} r(x)=\lim_{n\to\infty} \frac{s(n,x)}{n}\end{equation}
If $\mu$ is ergodic then this number is the same for a.a $x$ and equals $\displaystyle{\inf_{n\ge 1}\int_X \frac{s(n,x)}{n}d\mu}$. For an ergodic $\mu$ we will call this number $r_\mu$.  The set of all \si-invariant  ergodic measures we denote as $\mathcal{M}$. The set of points $x\in X$  for which limit $(\ref{e1})$ exists we call regular and denote as $\mathcal{R}$

 Of course, there could be points for which the limit in $(\ref{e1})$  does not exist.
 
 We  can also consider numbers $s_n=\displaystyle{\max_x s(n,x)}$.   It is a subadditive sequence of numbers $s_{n+m}\le s_n+s_m$ and we denote as $r$ the following number:
 \begin{equation}\label{e2}\displaystyle{r=\lim_{n\to\infty}\frac{s_n}{n}=\inf_{n\ge 1} \frac{s_n}{n}}
 \end{equation}
 It is known (see \cite{S}) that if $\si$ is continuous and $X$ is compact then
\begin{equation}\label{e3} r=\sup_{x\in\mathcal{R}} r(x)=\sup_{\mu\in\mathcal{M}} r_\mu\end{equation} 
For a periodic point $p=\si^k p$ we denote  as $r_{p}$ the following quantity 
$$r_{p}=\frac{ s(k,p)}{k}$$
It is easy to see that $r(p)$ exists (but can be $-\infty$) and $r(p)\le r_p$.

 We will show that if $\si$ has a closing property we can prove that:
\begin{theorem}\label{t3} Let $X$ be a compact metric space, $\si:X\to X$  a homeomorphism with closing property.  We denote as $\mathcal{P}$ the set of all periodic points. If $a:X\to B^\times $ is an  $\alpha$-H\"older continuous function, $a(n,x)$ is a cocycle generated by it and ${s(n,x)=\ln \|a(n,x)\| }$ then 
 \begin{equation}\label{e3} r=\sup_{x\in\mathcal{R}} r(x)=\sup_{\mu\in\mathcal{M}} r_\mu\le \sup_{p\in\mathcal{P}}r_p \end{equation}
 
 \end{theorem}
 
 An easy corollary of this theorem is the following important for us result.
\begin{corollary} \label{c3} Let $X$ be a compact metric space, $\si:X\to X$  a  homeomorphism with closing property. If $a(n,p)=e$ for every periodic point $p$ with period $n$ then for any $\ee>0$ there exists $C$ such that for all integer positive $n$ and all $x\in X$
$$\|a(n,x)\|\le Ce^{\ee n}$$
$$\|a(-n,x)\|\le Ce^{\ee n}$$
$$\|[a(n,x)]^{-1}\|\le Ce^{\ee n}$$
\end{corollary}
\begin{proof} The first inequality follows from the fact that if $s(n,x)=\ln\|a(n,x)\|$ then for this subadditive cocycle $r_p=0$ for every periodic point $p$ and from (\ref{e3}) follows that $r=0$. For the second inequality we can consider a cocycle $b(n,x)$ over $\si^{-1}$ generated by $a^{-1}(x)$.  Below, we will prove that  if $a(x)$ is H\"{o}lder continuous then $a^{-1}(x)$ is also H\"{o}lder continuous, and therefore we can apply Theorem \ref{t3} to the cocycle $b(n,x)$ also. But $a(-n,x)=b(n,x)$ and if $a(n,p)=e$ for every periodic point then 
$$b(n,p)=a(-n,p)=a(-n,\si^np)=[a(n,p)]^{-1}=e$$
So the rate of growth $r$ for $b(n,x)$ is also 0. 

The last inequality follows from the fact that
$$[a(n,x)]^{-1}=b(n,\si^n x)$$

The only thing left to show is that if $a(x)$ is $\alpha$-H\"{o}lder continuous then $a^{-1}(x)$ is also $\alpha$-H\"{o}lder continuous. For normed rings the operation of taking the inverse element is continuous (see \cite{Na}). Therefore the function $a^{-1}(x)$ is bounded. But 
$$\|a^{-1}-b^{-1}\|=\|b^{-1}(b-a)a^{-1}\|\le\|b^{-1}\|\cdot\|(b-a)\|\cdot\|a^{-1}\|$$
Therefore, if the function $a(x)$ is $\alpha$-H\"{o}lder continuous, then $a^{-1}(x)$ is also $\alpha$-H\"{o}lder continuous.
\end{proof}

\section{Proof of the Theorem \ref{t3}}

The following result proven in \cite[Proposition 4.2]{MK} will be used. 
\begin{lemman}[A. Karlsson, G. A. Margulis]\label{MK} Let $\si:X\to X$ be a measurable map, $\mu$ an ergodic measure, $s(n,x)$ a subadditive cocycle. For any $\ep>0$, let $E_\ep$ be the set of $x$ in $X$ for which there exist an integer $K(x)$ and infinitely many $n$ such that 
$$s(n,x)-s(n-k,\si^kx)\ge (r_\mu-\ep)k$$
for all $k, K(x)\le k\le n$. Let $E=\cap_{\ep>0} E_{\ep}$ then $\mu(E)=1$.
\end{lemman}
If $s(n,x)=\ln\|a(n,x)\|$ then the inequality in the lemma could be rewritten as 
\begin{equation}\|a(n-k,\si^k x)\|\le \|a(n,x)\|e^{-(r_\mu-\ep)k}\label{fmk}\end{equation}
\begin{definition} Let $\gamma,\delta$ be some positive numbers and $n$ is a natural number. We say that a point $y$ is $(\gamma,\delta,n)-$close to $x$ if 
$$dist(\si^k x,\si^k y)\le \delta e^{-\gamma k} \quad \text{for all}\quad 0\le k\le n$$
\end{definition} 

\begin{prop} \label{l5} Let  $\si:X\to X$ be a  homeomorphism  and $a:X\to B^\times$ be an $\alpha$-H\"older continuous  function, and $s(n,x)=\ln\|a(n,x)\|$. For any $\gamma>0$ let $S_\gamma$ be the set of points $x$ in $X$ for which there exist a number $\delta(x)>0$ and infinitely many $n$ such that for any point $y$ which is $(\gamma,\delta,n)-$close to $x$ 
\begin{equation}\label{f11}
\|a(n,y)\|\ge \frac12\|a(n,x)\|
\end{equation}
Then $\mu(S_\gamma)=1$ for any ergodic invariant measure $\mu$ with $r-r_\mu<\alpha\gamma$.
 \end{prop}
 \begin{proof}
Let $\mu$ be an ergodic invariant measure with $r-r_\mu<\alpha\gamma$. We choose a number $0<\ee<\frac13(\gamma\alpha-(r-r_\mu))$.  Almost all points with respect to this measure and \ee\  satisfy Karlsson-Margulis Lemma  and for almost all points the number $r(x)=r_\mu$. Take a point $x$ from the intersection of those two sets. Using the identity 
$$b_na_{n-1}\ldots b_1-a_na_{n-1}\ldots a_1=\sum_{k=1}^n b_n\ldots b_{k+1}(b_{k}-a_{k})a_{k-1}\ldots a_1$$
we can see that
$$\|a(n,x)-a(n,y)\|=\|\sum_{k=0}^{n-1} a(n-k-1,\si^{k+1} x)[a(\si^{k}x)-a(\si^{k}y)]a(k,y)\|\le$$
\begin{equation}\label{f4} 
\sum_{k=0}^{n-1} \|a(n-k-1,\si^{k+1} x)\|\cdot\|a(\si^{k}x)-a(\si^{k}y)\|\cdot \|a(k,y)\|
\end{equation}
Our goal is to show that if we choose  a sufficiently small  $\delta$ then for infinitely many numbers $n$  and for every   $(\gamma,\delta,n)-$close $y$ the sum $(\ref{f4})$ is smaller than $ \frac12\|a(n,x)\|$.
 Let $K(x,\ee)$ and $n$ be as in the Karlsson-Margulis Lemma and a point $y$ is $(\gamma,\delta,n)-$close to $x$ for some $\delta$ that we specify later. By definition $\displaystyle{r=\lim_{k\to\infty} s_k/k}$, so we  can find $K\ge K(x,\ee)$ such that $s_{k}<k(r+\ee)$ for all $k\ge K$, or $\|a(k,x)\|<e^{k(r+\ee)}$.  For every $k> K$ factors in the product 
\begin{equation}\label{f3} \|a(n-k-1,\si^{k+1} x)\|\cdot\|a(\si^{k}x)-a(\si^{k}y)\|\cdot \|a(k,y)\|
\end{equation} could be bounded from above as \\
$\|a(n-k-1,\si^k x)\|\le  \|a(n,x)\|e^{- (r_\mu-\epsilon)(k+1)}\quad$  It follows from the Karlsson-Margulis Lemma. \\
$\|a(\si^{k}x)-a(\si^{k}y)\|\le H\delta^\alpha e^{-k\gamma\alpha}$  where $H$ is some positive constant. It follows from the fact that $a(x)$ is H\"older continuous and $y$ is $(\gamma,\delta,n)$-close to $x$.\\
$ \|a(k,y)\|\le e^{s_{k}}\le e^{k(r+\ee)}$    It follows from the definition of $K$.\\
If we combine those inequalities we can see that that the number in the product  (\ref{f3}) is smaller than
$$ \|a(n,x)\|e^{- (r_\mu-\epsilon)(k+1)}\cdot H\delta^\alpha e^{-k\gamma\alpha}\cdot e^{k(r+\ee)}\le  \|a(n,x)\|H\delta^\alpha e^{-k(\gamma\alpha-(r-r_\mu)-2\ee)}$$  

After simplification we can write that
$$ \|a(n-k-1,\si^{k+1} x)\|\cdot\|a(\si^{k}x)-a(\si^{k}y)\|\cdot \|a(k,y)\|\le \|a(n,x)\|H\delta^\alpha e^{-k\ee}
$$

If we add  those inequalities for $k\ge K$ we can see that
$$\|\sum_{k=K}^n a(n-k+1,\si^{k+1} x)[a(\si^{k}x)-a(\si^{k}y)]a(k,y)\|\le$$
$$ \|a(n,x)\|H\delta^\alpha \sum_{k=0}^\infty e^{-k\ee}= \|a(n,x)\|\frac{H\delta^\alpha}{1-e^{-\ee}}$$
To estimate the number in the formula (\ref{f3}) for $k<K$ we denote as 
$$M=1+\max_x ||a(x)||$$ 
$$m=1+\max_x ||a^{-1}(x)||$$
 Then as before 
 $$\|a(\si^{k}x)-a(\si^{k}y)\|\le H \delta^\alpha e^{-k\gamma\alpha}$$
 but 
$$||a(n-k+1,\si^{k+1} x)||=||a(-k+1,\si^{n}x)a(n,x)||\le ||a(n,x)||\cdot m^k$$
and
$$||a(k,y)||\le M^{k}$$
So for  $k<K$ the expression (\ref{f3}) is bounded by 
$$||a(n,x)|| m^k\cdot H\delta^\alpha e^{-k\gamma\alpha}\cdot M^{k}\le ||a(n,x)||H\delta^\alpha (mM)^K$$
Finally, 
$$||a(n,x)-a(n,y)||\le ||a(n,x)||\delta^\alpha \left(\frac{H}{1-e^{-\ee}}+K(mM)^K\right)=||a(n,x)||\delta'$$
By choosing $\delta$ sufficiently small we can make $\delta'<1/2$. Then 
$$||a(n,y)||= ||a(n,x)-(a(n,x)-a(n,y))||\ge ||a(n,x)||-||a(n,x)-a(n,y)||\ge$$
$$\ge \frac12|| a(n,x)||$$
\end{proof}

To finish the proof of the Theorem \ref{t3} we  will need the following  features of the maps with closing property.
\begin{lemma}\label{l6} Let $\si:X\to X$ be a homeomorphism with closing property and the expansion constant $\la$, then for any positive numbers $\ee$ and $\delta$   there is a number $\delta'$ such that if $dist(x,\si^kx)\le\delta'$ and $k\ge n(1+\ee)$ then there is a point $p$ such that $\si^k p=p$ and $p$ is $(\gamma,\delta,n)-$close to $x$, where $\gamma=\ee\lambda$.
 \end{lemma}
\begin{proof} It follows from the definition of the closing property that for $0\le i\le k$
 $$dist(\si^i x,\si^i p)\le dist(\si^i x,\si^i z)+dist(\si^i z,\si^i p)\le 2C\delta' e^{-\lambda \min(i,k-i) }$$ 
 
 The function $-\lambda\min(x,k-x)$ is convex downward so the segment connecting points $(0,0)$ and $(n,-\lambda\min(n,k-n))$ on the graph of this function stays above the graph. The linear function that corresponds to this segment is $-\gamma x$ where 
 $$\gamma=\frac{k-n}{n}\lambda>\ee\lambda$$ 
 Therefore the point $p$ satisfies the following inequalities:
 $$dist(\si^i x,\si^i p)\le 2 C\delta' e^{-\gamma i } \quad 0\le i\le n$$ 
 If we take $\delta'=\frac{\delta}{2 C}$ we can see that $p$ is  $(\gamma,\delta,n)-$close to $x$.
 \end{proof}
\begin{lemma}\label{l7} Let $\si:X\to X$ be a homeomorphism.  For any $\ee,\delta>0$ let  $P_{\ep,\delta}$ be the set of  points $x$ in $X$ for which there is an integer number $N=N(x,\ep,\delta)$ such that if $n>N$ then there is an integer $ n(1+\ee)<k<n(1+2\ee)$  for which
$$dist(x,\si^k x)<\delta$$
If $\displaystyle{P=\cap_{\ee>0,\delta>0} P_{\ee,\delta}}$ then $\mu(P)=1$ for any invariant  measure $\mu$. 
\end{lemma}
\begin{proof} It is enough to prove it only for ergodic invariant measures. Let $\mu$ be an invariant ergodic measure. The support of a measure is the the set of all points in $X$ such that the measure of any open ball centered at $x$ is not 0. The support of a measure on a compact metric space  has always full measure {(see \cite{Fe})}. $X$ is compact so there is a sequence of balls $B_i$ which is  a base of the topology.  If we define as $f_i(n,x)$ the number  of such $k$  that $\si^k x\in B_i$ and $1\le k\le n$ then by Birkhoff's Ergodic Theorem $\displaystyle{\lim_{n\to\infty} \frac{f_i(n,x)}{n}}$ exists and equals $\mu(B_i)$ for almost all $x$. It is easy to see that any $x$ that belongs to the support of the measure and  satisfies  Birkhoff's Ergodic Theorem for all $i$ will belong to the set $P$. Indeed, if we choose $\delta>0$ then we know that the ball $B_\delta$ centered at $x$ should have measure greater than 0. This ball is a countable union of some of the balls $B_i$, therefore there exists at least one ball $B_{i_0}$ such that $\mu(B_{i_0})>0$ and $B_{i_0}\subset B_\delta$. Now, using the numbers $\ee$ and $\mu(B_{i_0})$ we choose a very small $\ep$. How small we specify later.  For this $\ep>0$ we can find $N$ such that if $n>N$ then $|f_{i_0}(n,x)-\mu(B_{i_0}n)|<\ep n$. If $n>N$ and there is no $k$ such that ${n(1+\ee)<k<n(a+2\ee)}$ and ${\si^k x\in B_{i_0}}$ then $f_{i_0}(n(1+\ee),x)=f_{i_0}(n(1+2\ee),x)$. It is impossible if we choose $\ep$ very small because in this case
$$ (\mu(B_{i_0})+\ep)n(1+\ee)\ge f_{i_0}(n(1+\ee,x)=f_{i_0}(n(1+2\ee,x)\ge (\mu(B_{i_0})-\ep)n(1+2\ee)$$
or
$$\frac{\mu(B_{i_0})+\ep}{\mu(B_{i_0})-\ep}\ge \frac{1+2\ee}{1+\ee}$$
When $\ep$ is small the left side is as close to 1 as we want, so we get a contradiction. It means, if $N$ is sufficiently big and $n>N$ then there is $k$ such that ${\si^k\in B_{i_0}\subset B_\delta }$ and ${n(1+\ee)\le k\le n(1+2\ee)}$. Therefore the set $P$ includes the intersection of two sets of full measure and  has the full measure.
\end{proof}
 {\it Proof of the Theorem \ref{t3}:}  Choose any $\ee>0$. We can find an ergodic invariant measure $\mu$ such that ${r-r_\mu<min(\ee,\ee\alpha\lambda)}$. Choose a point $x$ such that $r(x)=r_\mu$ and $x$ belongs to the set $S_{\ee\lambda}\cap P$ where $S_{\ee\lambda}$ as in the Proposition \ref{l5} and $P$ as in the Lemma \ref{l7}. All those sets have full support, so their intersection is not empty. For the point $x$ we can find $\delta$ such that for infinitely many $n_i$ if a point $p$ is $(\ee\lambda,\delta,n_i)-$close to $x$ then
 \begin{equation}\label{f10}\|a(n_i,p)\|\ge\frac12\|a(n_i,x)\|\end{equation}
For this $\delta$ we can find $\delta'$ from Lemma \ref{l6}. Using this  $\delta'$ and $\ee$ we can find $N=N(\ee,\delta')$ from the Lemma \ref{l7} such that if $n_i>N$ then there is $k$ such that $n_i(1+\ee)\le k\le n_i(1+2\ee)$ and $dist(\si^kx,x)<\delta'$, then from Lemma \ref{l6} follows that there is a periodic point $p$ with the period $k$ such that it is $(\ee\lambda,\delta,n_i)-$close to $x$ and therefore satisfies the {inequality (\ref{f10})}.

Now, we estimate $\|a(k,p)\|$. Let $N'$ be a number such that if $n>N'$ then 
$$\|a(n,x)\|\ge e^{n(r_\mu-\ee)}\ge e^{n(r-2\ee)}$$
We always can choose $n_i$ bigger not only than $N$ but also and $N'$. Denote as ${m=\ln\max_y\|a^{-1}(y)\|}$. Then
$$\|a(n_i,p)\|=\|a(-(k-n_i),p)a(k,p)\|\le \|a(-(k-n_i),p)\|\cdot\|a(k,p)\|$$
 so
 $$\|a(k,p)\|\ge \frac{\|a(n_i,p)\|}{e^{m(k-n_i)}}\ge\frac12\frac{\|a(n_i,x)\|}{e^{2m\ee n_i}}\ge \frac12 e^{(r-2\ee-2m\ee)n_i}$$
 We see that
 $$r_p=\frac{\ln\|a(k,p)\|}{k}\ge\frac{(r-2\ee-2m\ee)n_i-\ln 2}{(1+2\ee)n_i}$$
 Number $m$ does not depend on the choice of $x,n_i$ and $p$, so by choosing $\ee$ very small and $n_i$ very big we can make $r_p$ as close to $r$ as we want.

 \qed\\

\section{Proof of the Main Theorem }

After Theorem \ref{t3} is established we can use Corollary \ref{c3} to show that  the growth of $\|a(n,x)\|$ is sub-exponential.  It allows to use the idea of the original Liv\v{s}ic proof for cocycles with values in Banach rings. 
 H.Bercovici and V.Nitica in \cite{BN} (Theorem 3.2) showed that if $\si$ is a transitive Anosov map, periodic obstructions vanish and  
 \begin{equation}\label{f11}\begin{split}\|a(x)\|&\le 1+\delta\\
 \|a^{-1}(x)\|&\le 1+\delta
 \end{split}\end{equation}
 for some $\delta$ that depends on $\si$, then $a(x)$ is a coboundary.  From Corollary \ref{c3} we can get a little bit less. If periodic obstructions vanish then for any $\delta>0$ there exists $C>0$ such that for any positive integer $n$
  \begin{equation*}\begin{split}\|a(n,x)\|&\le C(1+\delta)^n\\
 \|[a(n,x)]^{-1}\|&\le C(1+\delta)^n
 \end{split}\end{equation*}
 Those inequalities are  actually enough to repeat the arguments from \cite{BN} with some small changes, but we also refer to more general theorem proven in \cite{G} that considers cocycles over maps that satisfy closing property and with values in  abstract groups satisfying some conditions. But we will need couple of more definitions.
\begin{definition} If $G$ is a group with metric denoted as $dist$ and $g\in G$ we define the distortion of the element $g$ as 
$$|g|=\sup_{f\neq g}\left[ \frac{dist(gf,gh)}{dist{(f,h)}},\frac{dist(fg,hg)}{dist{(f,h)}},\frac{dist(g^{-1}f,g^{-1}h)}{dist{(f,h)}},\frac{dist(fg^{-1},hg^{-1})}{dist{(f,h)}}\right]$$
We say that a group is Lipschitz if $|g|<\infty$ for all $g\in G$.
\end{definition}
It is easy to see that for Banach rings if we define $$dist(f,h)=\max(\|f-h\|,\|f^{-1}-g^{-1}\|)$$ then  $|g|\le \max(\|g\|,\|g^{-1}\|)$ and $B^\times$ is Lipschitz.
\begin{definition} We call the rate of distortion of a cocycle $a(x):X\to G$ the following number
$$r(a)=\lim_{n\to\infty} \frac{\sup_{x\in X}\ln |a(n,x)|}{n}$$
\end{definition}
\begin{theorem}\label{t9}  Let $G$ be a Lipschitz group with the property that there are numbers $\ep$ and $D$ such that $dist(g,e)\le\ep$ implies $|g|\le D$, \si\ be a transitive homeomorphism with $\lambda$-closing property. If the rate of the distortion of an $\alpha$-H\"older continuous cocycle $a:X\to G$ is smaller than $\alpha\la/2$ and the periodic obstructions vanish then $a(x)$ is a coboundary with $\alpha$-H\"older continuous transition function $t(x)$.
\end{theorem}
\begin{proof} See \cite{G}\end{proof}
{\it Proof of the Main Theorem.} If $a(n,p)=e$ for every periodic point then  it follows from the Corollary \ref{c3} that the distortion rate of the $a(n,x)$ is less or equal than 0. In the  group $B^\times$ if $dist(e,g)<\frac12$ then $\|e-g\|<\frac12$ and $\|e-g^{-1}\|<\frac12$ so  $\|g\|,\|g^{-1}\|\le\frac32$. We see that $|g|<\frac32$, therefore by Theorem \ref{t9} the cocycle $a(n,x)$ is a coboundary with 
$\alpha$-H\"older continuous transition function.\qed

%


\begin{thebibliography}{99}

\bibitem{Furstenberg}
H.~Furstenberg and H.~Kesten.
\emph{Products of random matrices}, The Annals of Mathematical Statistics, 31(2):457--469, June 1960.

\bibitem{BN} H. ~Bercovici, V. ~Nitica
 \emph{ A Banach Algebra Version of the Livsic Theorem},
   Discr. Contin. Dyn. Syst., \textbf{4} (1998), no. 3 523--534.
\bibitem{Fe} H.~Federer,\emph{ Geometric Measure Theory }, Springer, New York, 1969
\bibitem{G} M.~Guysinsky,
 \emph{ Liv\v{s}ic Theorem for  Cocycles with Values in the Group of Diffeomorphisms }, preprint, (2013).
\bibitem{Ka} B.~Kalinin,
 \emph{ Liv\v{s}ic Theorem for matrix cocycles },
Annals of Mathematics  \textbf{173} (2011), 1025--1042.

\bibitem{Ki} J.F.C. ~Kingman,
 \emph{ The ergodic theory of subadditive stochastic processes. },
 J. Roy. Statist. Soc. Ser. B \textbf{30} (1968), 499--510. 
\bibitem{L1} A. ~Liv\v{s}ic,
 \emph{ Homology properties of Y-systems},
   Math. Zametki, \textbf{10} (1971), 555--564.
   
\bibitem{L2} A. ~Liv\v{s}ic,
 \emph{Cohomology of dynamical systems},
 Izv. Akad. Nauk SSSR Ser. Mat.\textbf{36} (1972), 1296--1320.
 
 \bibitem{P} W.~Parry,
   \emph{The Liv\v{s}ic periodic point theorem for non-abelian cocycles},
    Ergodic Theory Dynam. Systems \textbf{19} (1999), no. 3, 687--701. 

\bibitem{KS} K.~Schmidt,
   \emph{Remarks on Liv\v{s}ic theory for nonabelian cocycles},
    Ergodic Theory Dynam. Systems \textbf{19} (1999), no. 3, 687--701. 
 
 \bibitem{S} S. J. ~Schreiber,\emph{ On growth rates of subadditive functions for semiflows}, J. 
Differential Equations 148 (1998), 334--350.

\bibitem{MK} A. ~Karlsson, G. A. ~Margulis,
\emph{A multiplicative ergodic theorem and nonpositively curved spaces},
Communications in mathematical physics, 208 (1999), 107--123.

\bibitem{KH} A.~Katok, B.~Hasseblatt, \emph{Introduction to the Modern Theory of Dynamical Systems}, Cambridge University Press, Cambridge,
1995.


\bibitem{LZ} L.~Zhu,\emph{ Liv\v{s}ic Theorem for Cocycles with Value in $GL(N,\mathbb{Q}_p$}, Ph.D. thesis, The Pennsylvania State University, (2012),1--53.

\bibitem{LW} R. de la Llave, A.~Windsor, \emph{Liv\v{s}ic theorems for non-commutative groups including groups of diffeomorphisms, and invariant geometric structures}, Ergodic Theory Dynam. Systems, \textbf{30} (2010), 1055Ð1100

\bibitem{Na} M. A.~Naimark, \emph{Normed Rings}, Groningen, Netherlands: P. Noordhoff N. V., 1959 
\bibitem{NT} V.~Nitica, A.~Torok, \emph{Cohomology of dynamical systems and rigidity of partially hyperbolic actions of higher-rank lattices}, Duke Math. J., \textbf{79} (1995), no. 3, 751--810.

\bibitem{PW} M.~Pollicott, C. P.~Walkden, \emph{Liv\v{s}ic theorems for connected Lie groups}, Trans. Amer. Math. Soc. \textbf{353} (2001), 2879-2895.
\end{thebibliography}
\end{document}